\documentclass[reqno]{amsart}

\usepackage{graphicx}
\usepackage{amscd}
\usepackage{amsmath, amssymb}
\usepackage{graphics}
\usepackage{graphicx}
\usepackage{epsfig}
\usepackage{xcolor}
\usepackage{soul}
\usepackage{tikz}
\usepackage{graphicx}
\usepackage{ mathrsfs }
\usepackage{ amssymb }
\usepackage{amscd}
\usepackage{amsthm }
\usepackage[utf8]{inputenc}
\usepackage{dirtytalk}
\usepackage{multirow}
\usepackage[all]{xy}

\usepackage{fancyhdr,lipsum}

\newtheorem{theorem}{Theorem}

\newtheorem{corollary}[theorem]{Corollary}

\newtheorem{definition}[theorem]{Definition}

\newtheorem{lemma}[theorem]{Lemma}

\newtheorem{proposition}[theorem]{Proposition}
\newtheorem{remark}[theorem]{Remark}

\newtheorem{question}[theorem]{Question}

\begin{document}

\title{On Universal Virtual and Welded Braid Groups and Their Linear Representations}
\author{Mohamad N. Nasser}
\address{Mohamad N. Nasser\\
         Department of Mathematics and Computer Science\\
         Beirut Arab University\\
         P.O. Box 11-5020, Beirut, Lebanon}        
\email{m.nasser@bau.edu.lb}

\author{Oscar Ocampo} 
\address{Oscar Ocampo\\  Departamento de Matem\'atica - IME\\ Universidade Federal da Bahia\\ Av.~Milton Santos~S/N, CEP:~40170-110 - Salvador - BA - Brazil}
\email{oscaro@ufba.br}



\date{\today}

\maketitle

\begin{abstract}
We introduce linear representations of the universal virtual braid group $UV_n(c)$, where $n\geq 2$ and $c\geq 1$, which is a unifying framework for braid-type groups with multiple types of crossings. We classify and study its complex homogeneous $2$-local representations for all $n\geq 3$ and $c\geq 1$ (unique up to equivalence) and complex homogeneous $3$-local representations for all $n\geq 4$ and $c=2$ (four distinct families). We then introduce the universal welded braid group $UW_n(c)$ as a quotient of $UV_n(c)$ by the welded relations. This group recovers all known welded-type groups as quotients. We prove that $UW_n(c)$ has abelianization $\mathbb{Z}^c \oplus \mathbb{Z}_2$, perfect commutator subgroup for $n \geq 5$, trivial center, and $S_n$ as its smallest non-abelian finite quotient. Finally, we classify and study the complex homogeneous $2$-local representations of $UW_n(c)$ for all $n\geq 3$ and $c\geq 1$, obtaining three distinct families.
\end{abstract}

\section{Introduction}

Braid groups and their generalizations form a fundamental class of groups in geometric and combinatorial group theory. Virtual braid groups were introduced by Kauffman \cite{K} in the context of virtual knot theory and further studied by Bardakov \cite{BSV}, Singh, and Vesnin. Subsequent extensions, including virtual singular braids \cite{CPM}, virtual twin groups \cite{BSV}, and multi-virtual braid groups \cite{Kau2}, enrich the crossing structure while retaining a Coxeter-type virtual symmetry. Despite their common features, these families are typically treated as separate constructions, which suggests the existence of a universal algebraic structure from which they arise as natural quotients. \vspace{0.1cm}

In a previous paper \cite{O2}, the second author introduced the \emph{universal virtual braid group} $UV_n(c)$, defined for $n \geq 2$ strands and $c \geq 1$ types of non-virtual crossings. This group encodes only the essential interactions between virtual crossings and multiple types of non-virtual crossings, and admits canonical quotient maps onto the virtual braid group $VB_n$, the virtual singular braid group $VSG_n$, the virtual twin group $VT_n$, and the multi-virtual braid group $M_kVB_n$ \cite[Proposition~2.2]{O2}. Moreover, $UV_n(c)$ contains a finite-index right-angled Artin subgroup $KUV_n(c)$ \cite[Theorem~2.10]{O2}, which yields strong structural consequences such as residual finiteness, linearity, and the Tits alternative. \vspace{0.1cm}

A key feature of the universal virtual braid group $UV_n(c)$ is that it provides a natural framework from which several meaningful quotient groups arise by imposing forbidden relations, in analogy with the classical theory of virtual braid groups. Among these, welded-type quotients play a distinguished role, as they preserve much of the underlying combinatorial structure while introducing relations that lead to significant structural simplifications. Motivated by this perspective, we introduce the \emph{universal welded braid group} $UW_n(c)$ (Definition~\ref{def:welded}) and investigate its algebraic properties, highlighting its role as a canonical quotient within the universal setting.  \vspace{0.1cm}


The representation theory of braid-type groups has been extensively studied. A particularly fruitful approach is the study of \emph{$k$-local representations}, initiated by Mikhalchishina \cite{Mik2013}, who classified the $2$-local representations of $B_3$ and all complex homogeneous $2$-local representations of $B_n$ for $n \geq 3$. This line of research was extended by Mayassi and Nasser \cite{Mayassi2025}, who classified complex homogeneous $3$-local representations of $B_n$ for $n \geq 4$, and also classified all complex homogeneous $2$-local and $3$-local representations of the singular braid group $SB_n$. Subsequently, Keshari, Nasser, and Prabhakar \cite{nassernew} classified all complex homogeneous $2$-local representations of the multi-virtual braid group $M_kVB_n$ and the multi-welded braid group $M_kWB_n$. Nasser further provided complete classifications for the virtual braid group $VB_n$ and the welded braid group $WB_n$ \cite{Nasser2026}, as well as for the twin group $T_n$, the virtual twin group $VT_n$, and the welded twin group $WT_n$ \cite{Mayasi20251, Nasser202622}. \vspace{0.1cm}

Building on this body of work, in the present paper we classify complex homogeneous $2$-local and $3$-local representations of $UV_n(c)$ and $UW_n(c)$. Our main results are as follows.

\begin{itemize}
    \item \textbf{Classification of $2$-local representations of $UV_n(c)$:} Every nontrivial complex homogeneous $2$-local representation of $UV_n(c)$ is, up to equivalence, uniquely determined for all $n\geq 3$ and $c\geq 1$, and has the form given in Theorem~\ref{thm2local}. Moreover, we give necessary and sufficient conditions for its irreducibility (Theorem~\ref{irred2loc}).
    
    \item \textbf{Classification of $3$-local representations of $UV_n(c)$ for $c=2$:} For $n \geq 4$ and $c = 2$, there are exactly four families of nontrivial complex homogeneous $3$-local representations (Theorem~\ref{thm3local}), all of which are reducible (Theorem~\ref{thm3localred}).
    
    \item \textbf{The universal welded braid group $UW_n(c)$:} We define $UW_n(c)$ as the quotient of $UV_n(c)$ by the welded relations (Definition~\ref{def:welded}). This group unifies the welded braid group $WB_n$, the welded singular braid group $WSG_n$, the welded twin group $WT_n$, and the multi-welded braid group $M_kWB_n$ as natural quotients, see Proposition~\ref{prop:welded_quotients}.
    
    \item \textbf{Properties of $UW_n(c)$:} We compute its abelianization (Proposition~\ref{prop:uw_abel}), prove that its commutator subgroup is perfect for $n \geq 5$ (Theorem~\ref{thm:uw_perfect}), determine its center (Proposition~\ref{prop:uw_center}), and show that the symmetric group $S_n$ is the smallest non-abelian finite quotient (Theorem~\ref{thm:uw_smallest_quotient}). As a consequence, the same minimality holds for all welded-type groups listed above (Corollary~\ref{cor:welded_perfect}  and Corollary~\ref{cor:w_smallest_quotient}).

    \item \textbf{$2$-local representations of $UW_n(c)$:} We classify all complex homogeneous $2$-local representations of $UW_n(c)$ for all $n\geq 3$ and $c\geq 1$, obtaining three distinct families (Theorem~\ref{thm2localUWn}), and analyze their irreducibility (Proposition~\ref{prop:w2red}).\vspace{0.1cm}
\end{itemize}

The paper is organized as follows. Section~\ref{sec:prelim} collects the necessary preliminaries: we recall the definition of $UV_n(c)$ and its natural quotients, and we introduce the notion of $k$-local representations. In Section~\ref{sec3} we classify and study the irreducibility of all $2$-local representations of $UV_n(c)$ for all $n\geq 3$ and $c\geq 1$. Section~\ref{3local} deals with the $3$-local case for $n\geq 4$ and $c = 2$. In Section~\ref{sec:welded} we define the universal welded braid group $UW_n(c)$, show that it recovers all known welded-type groups, and establish some structural properties. Finally, Section~\ref{sec6} classifies and studies the irreducibility of the $2$-local representations of $UW_n(c)$ for all $n\geq 3$ and $c\geq 1$.

\subsection*{Acknowledgments}

The second named author gratefully acknowledges the personal and medical support of Eliane Santos, the staff of HCA, Bruno Noronha, Luciano Macedo, M\'arcio Isabella, Andreia de Oliveira Rocha, Andreia Gracielle Santana, Ednice de Souza Santos, and SMURB--UFBA (Servi\c{c}o M\'edico Universit\'ario Rubens Brasil Soares), whose support since July 2024 was essential in enabling the completion of this work. O.~O.~was partially supported by the National Council for Scientific and Technological Development (CNPq, Brazil) through a \textit{Bolsa de Produtividade} grant No.~305422/2022--7.

\vspace*{0.1 cm}

\section{Preliminaries and Previous Results}\label{sec:prelim}

In this section we collect the necessary background and fix the notation that will be used throughout the paper. We first recall the definition of the universal virtual braid group $UV_n(c)$ and describe its natural quotients, which include the virtual braid group, the virtual singular braid group, the virtual twin group, and the multi-virtual braid group. We then introduce the concept of $k$-local representations, a key tool for the classification results presented in Sections~\ref{sec3}, \ref{3local}, and \ref{sec6}, and provide a brief overview of the relevant literature.

\subsection{The universal virtual braid group and its quotients}

We begin by recalling the definition of the universal virtual braid group, which serves as the central object of this paper. This group unifies several families of virtual braid-type groups by encoding only the essential interactions between virtual crossings and multiple types of non-virtual crossings.

\begin{definition}[Universal virtual braid group \cite{O2}] 
Let $n \geq 2$ and $c \geq 1$. The universal virtual braid group $UV_n(c)$ is the group given by the presentation with:
\begin{itemize}
    \item \textbf{Generators:} $\rho_i$ for $i = 1,2,\dots,n-1$ and $\sigma_{i,t}$ for $i = 1,2,\dots,n-1$ and $t = 1,2,\dots,c$.
    \item \textbf{Relations:}
    \begin{align*}
        &\text{(PR1)}\quad \rho_i\rho_{i+1}\rho_i = \rho_{i+1}\rho_i\rho_{i+1}, \qquad i = 1,2,\dots,n-2,\\
        &\text{(PR2)}\quad \rho_i\rho_j = \rho_j\rho_i, \qquad |i-j| \geq 2,\\
        &\text{(PR3)}\quad \rho_i^2 = 1, \qquad i = 1,2,\dots,n-1,\\
        &\text{(CR)}\quad \sigma_{i,t}\sigma_{j,\ell} = \sigma_{j,\ell}\sigma_{i,t}, \qquad |i-j| \geq 2,\; 1 \leq t,\ell \leq c,\\
        &\text{(MR1)}\quad \sigma_{i,t}\rho_j = \rho_j\sigma_{i,t}, \qquad |i-j| \geq 2,\; 1 \leq t \leq c,\\
        &\text{(MR2)}\quad \rho_i\rho_{i+1}\sigma_{i,t} = \sigma_{i+1,t}\rho_i\rho_{i+1}, \qquad i = 1,2,\dots,n-2,\; 1 \leq t \leq c.
    \end{align*}
\end{itemize}
By convention, for any $c\geq 1$, the group $UV_1(c)$ is trivial.
\end{definition}

For a geometric example of a braid with 4 strands and $c=2$ types of non-virtual crossings, see \cite[Figure~1]{O2}.
For $n=2$, the presentation simplifies (see \cite[Proposition~3.1(1)]{O2}): there is only one virtual generator $\rho_1$ (with $\rho_1^2=1$) and $c$ non-virtual generators $\sigma_{1,1},\ldots,\sigma_{1,c}$ with no relations among them. Hence $UV_2(c) \cong F_c * \mathbb{Z}_2$, where $F_c$ is the free group of rank $c$. 

As shown in \cite[Proposition~2.2]{O2}, the following virtual braid-type groups arise as quotients of $UV_n(c)$:
\[ 
    \xymatrix@C=3.5pc{
UV_n(k) \ar@{->>}[d] \ar@{->>}[r] & UV_n(2) \ar@{->>}[d] \ar@{->>}[r] & UV_n(1) \ar@{->>}[d] \ar@{->>}[rd] & \\
M_kVB_n    & VSG_n \ar@{->>}[r] & VB_n & VT_n
}
 \]
where $VB_n$ is the virtual braid group \cite{K}, $VT_n$ is the virtual twin group \cite{BSV}, $VSG_n$ is the virtual singular braid group \cite{CPM}, and $M_kVB_n$ is the multi-virtual braid group \cite{Kau2}.

Let $S_n$ denote the symmetric group on $n$ letters, generated by the adjacent transpositions $s_i = (i\ i+1)$ for $i = 1,2,\dots,n-1$, subject to the usual Coxeter relations. In \cite[Section~2.2]{O2}, two natural surjective homomorphisms from $UV_n(c)$ onto $S_n$ were introduced:
\[
\pi_n^P\colon UV_n(c)\to S_n,
\qquad
\pi_n^K\colon UV_n(c)\to S_n,
\]
defined by
\[
\pi_n^P(\sigma_{i,t}) = s_i,\quad \pi_n^P(\rho_i) = s_i,
\qquad
\pi_n^K(\sigma_{i,t}) = 1,\quad \pi_n^K(\rho_i) = s_i,
\]
for all $i = 1,2,\dots,n-1$ and $t = 1,2,\dots,c$. Their kernels are denoted by
\[
PUV_n(c) = \ker(\pi_n^P)
\quad\text{and}\quad
KUV_n(c) = \ker(\pi_n^K).
\]
The group $PUV_n(c)$ is called the \emph{pure universal virtual braid group}. The subgroup $KUV_n(c)$ is a right-angled Artin group (RAAG); see \cite[Theorem~2.10]{O2}.

Moreover, for all $n \geq 2$ and $c \geq 1$, the maps $\pi_n^P$ and $\pi_n^K$ split via the inclusion $\iota\colon S_n \hookrightarrow UV_n(c)$ sending $s_i \mapsto \rho_i$. Consequently, $UV_n(c)$ admits the semidirect product decompositions:
\[
UV_n(c) \cong PUV_n(c) \rtimes S_n
\quad\text{and}\quad
UV_n(c) \cong KUV_n(c) \rtimes S_n.
\]


\subsection{$k$-local representations: definition and background}

In what follows, we present the notion of $k$-local representations for a finitely generated group $G$.

\vspace*{0.1cm}

\begin{definition}\cite{Nas20241}
Let $G$ be a group generated by $g_1, g_2, \ldots, g_{n-1}$. A representation $\theta: G \to \mathrm{GL}_m(\mathbb{C})$ is called \emph{$k$-local} if
\[
\theta(g_i) =
\begin{pmatrix}
I_{i-1} & 0 & 0 \\
0 & M_i & 0 \\
0 & 0 & I_{n-i-1}
\end{pmatrix}, \quad 1 \leq i \leq n-1,
\]
where $M_i \in \mathrm{GL}_k(\mathbb{C})$, $k = m - n + 2$, and $I_r$ denotes the identity matrix of size $r \times r$. The representation is said to be \emph{homogeneous} if $M_i = M_j$ for all $1\leq i,j \leq n-1$.
\end{definition}

\begin{remark}
The concept of $k$-local representations extends naturally to groups generated by $m(n-1)$ elements, partitioned into $m$ disjoint families. For simplicity, we introduce in the following the case $m=2$. Suppose that $G$ is a group generated by
$
\{g_i, h_i \mid 1 \leq i \leq n-1\}.$ 
A $k$-local representation $\theta: G \to \mathrm{GL}_m(\mathbb{C})$ is defined by
\[
\theta(g_i) =
\begin{pmatrix}
I_{i-1} & 0 & 0 \\
0 & G_i & 0 \\
0 & 0 & I_{n-i-1}
\end{pmatrix} \quad \text{ and } \quad 
\theta(h_i) =
\begin{pmatrix}
I_{i-1} & 0 & 0 \\
0 & H_i & 0 \\
0 & 0 & I_{n-i-1}
\end{pmatrix}, \quad 1 \leq i \leq n-1,
\]
where $G_i, H_i \in \mathrm{GL}_k(\mathbb{C})$ and $k = m - n + 2$. The representation $\theta$ is homogeneous if $G_i = G_j$ and $H_i = H_j$ for all $1\leq i,j \leq n-1$.
\end{remark}

We now present famous examples of homogeneous $k$-local representations of the braid group $B_n$ of different degrees $k$.

\begin{definition}\cite{Burau}\label{defBurau}
The complex specialization of Burau representation $\Psi_B: B_n \to \mathrm{GL}_n(\mathbb{C})$ is defined by
\[
\sigma_i \mapsto
\begin{pmatrix}
I_{i-1} & 0 & 0 \\
0 &
\begin{pmatrix}
1 - t & t \\
1 & 0
\end{pmatrix}
& 0 \\
0 & 0 & I_{n-i-1}
\end{pmatrix}, \quad 1 \leq i \leq n-1.
\]
\end{definition}

\begin{definition}\cite{19Bar}\label{Fdef}
The complex specialization of the  $F$-representation $\Psi_F: B_n \to \mathrm{GL}_{n+1}(\mathbb{C})$ is defined by
\[
\sigma_i \mapsto
\begin{pmatrix}
I_{i-1} & 0 & 0 \\
0 &
\begin{pmatrix}
1 & 1 & 0 \\
0 & -t & 0 \\
0 & t & 1
\end{pmatrix}
& 0 \\
0 & 0 & I_{n-i-1}
\end{pmatrix}, \quad 1 \leq i \leq n-1.
\]
\end{definition}

In recent years, $k$-local representations have attracted increasing attention. Mikhalchishina initiated this line of research by classifying all $2$-local representations of $B_3$ and all complex homogeneous $2$-local representations of $B_n$ for $n \geq 3$ \cite{Mik2013}. This work was later extended by Mayassi and Nasser, who studied complex homogeneous $3$-local representations of $B_n$ for $n \geq 4$; in the same work, they also classified all complex homogeneous $2$-local and $3$-local representations of the singular braid group $SB_n$ for $n \geq 2$ and $n \geq 4$, respectively \cite{Mayassi2025}. Subsequently, V.~Keshare et al. classified all complex homogeneous $2$-local representations of the multi-virtual braid group $M_nVB_n$ and the multi-welded braid group $M_kWB_n$ for $n \geq 3$ and $k \geq 1$ \cite{nassernew}. Moreover, Nasser obtained a complete classification of complex homogeneous $2$-local and $3$-local representations of the virtual braid group $VB_n$ and the welded braid group $WB_n$ for $n \geq 2$ and $n \geq 4$, respectively \cite{Nasser2026}.In a different direction, Mayassi and Nasser determined all homogeneous $2$-local representations of the twin group $T_n$ over $\mathbb{Z}[t^{\pm1}]$ \cite{Mayasi20251}. More recently, Nasser classified all complex homogeneous $3$-local representations of the twin group $T_n$, the virtual twin group $VT_n$, and the welded twin group $WT_n$ for $n \geq 4$ \cite{Nasser202622}.

\section{Construction and Irreducibility of Complex Homogeneous $2$-Local Representations of $UV_n(c)$} \label{sec3}

In this section, we classify and study all complex homogeneous $2$-local representations of the group $UV_n(c)$ for all $n \geq 3$ and $c \geq 1$. These values of $n$ and $c$ will be assumed throughout this section.

\begin{theorem} \label{thm2local}
Let $\upsilon: UV_n(c) \longrightarrow \mathrm{GL}_n(\mathbb{C})$ be a nontrivial homogeneous $2$-local representation. Then, up to equivalence of representations, $\upsilon$ is uniquely determined as follows.
\[
\upsilon(\rho_i)=
\begin{pmatrix}
I_{i-1} & 0 & 0 \\
0 & \begin{pmatrix} 0 & r_2 \\ \dfrac{1}{r_2} & 0 \end{pmatrix} & 0 \\
0 & 0 & I_{n-i-1}
\end{pmatrix}
\text{ \ and \ }
\upsilon(\sigma_{i,t})=
\begin{pmatrix}
I_{i-1} & 0 & 0 \\
0 & \begin{pmatrix} s_{1,t} & s_{2,t} \\ 
s_{3,t} & s_{4,t} \end{pmatrix} & 0 \\
0 & 0 & I_{n-i-1}
\end{pmatrix}
\]
for all $1 \leq i \leq n-1$ and $1 \leq t \leq c$, where $r_2, s_{1,t}, s_{2,t}, s_{3,t}, s_{4,t} \in \mathbb{C}$ satisfy
\[
r_2\neq 0 \quad \text{and} \quad s_{1,t} s_{4,t} - s_{2,t} s_{3,t} \neq 0.
\]
\end{theorem}

\begin{proof}
Since $\upsilon$ is a complex homogeneous $2$-local representation of $UV_n(c)$, we may write
\[
\upsilon(\rho_i)=
\begin{pmatrix}
I_{i-1} & 0 & 0 \\
0 & \begin{pmatrix} r_1 & r_2 \\ 
r_3 & r_4 \end{pmatrix} & 0 \\
0 & 0 & I_{n-i-1}
\end{pmatrix} \text{ and }
\upsilon(\sigma_{i,t})=
\begin{pmatrix}
I_{i-1} & 0 & 0 \\
0 & \begin{pmatrix} s_{1,t} & s_{2,t} \\ 
s_{3,t} & s_{4,t} \end{pmatrix} & 0 \\
0 & 0 & I_{n-i-1}
\end{pmatrix}
\]
for all $1 \leq i \leq n-1$ and $1 \leq t \leq c$, where $r_1, r_2, r_3, r_4, s_{1,t}, s_{2,t}, s_{3,t}, s_{4,t} \in \mathbb{C}$ satisfy
\[
r_1 r_4 - r_2 r_3 \neq 0 \text{\ \  and\ \  } s_{1,t} s_{4,t} - s_{2,t} s_{3,t} \neq 0.
\]
Now, as $\upsilon$ is a representation of $UV_n(c)$, it must preserve the defining relations of the group. In addition, since $\upsilon$ is $2$-local, it suffices to verify only a subset of these relations, because the remaining ones automatically yield analogous conditions. In particular, it is enough to consider the relations
\[
\rho_1 \rho_2 \rho_1 = \rho_2 \rho_1 \rho_2, \qquad
\rho_1^2 = 1, \qquad
\rho_1 \rho_2 \sigma_{1,1} = \sigma_{2,1} \rho_1 \rho_2.
\]
Here we focus on a single family of the generators $\sigma_{i,t}$, namely $\sigma_{i,1}$, since for $2$-local representations the relations involving the generators $\sigma_{i,t}$ together are automatically satisfied. The same argument applies to all $\sigma_{i,t}$ with $2 \le t \le c$. Applying these relations via matrix multiplication then gives a system of equations in the entries of the matrices. After eliminating the repeated equations, we obtain a system of fifteen equations in eight unknowns, which is presented below.
\begin{equation} \label{2loc1}
r_1(1-r_1-r_2r_3)=0
\end{equation}
\begin{equation}\label{2loc2}
r_1r_2r_4=0
\end{equation}
\begin{equation}\label{2loc3}
r_1r_3r_4=0
\end{equation}
\begin{equation}\label{2loc4}
r_1r_4(r_1-r_4)=0
\end{equation}
\begin{equation}\label{2loc5}
r_4(1-r_2r_3-r_4)=0
\end{equation}
\begin{equation}\label{2loc6}
-1+r_1^2+r_2r_3=0
\end{equation}
\begin{equation}\label{2loc7}
r_2(r_1+ r_4)=0
\end{equation}
\begin{equation}\label{2loc8}
r_3(r_1+ r_4)=0
\end{equation}
\begin{equation}\label{2loc9}
-1+r_2r_3+r_4^2=0
\end{equation}
\begin{equation}\label{2loc10}
r_1(-1+s_{1,1}+r_2s_{3,1})=0
\end{equation}
\begin{equation}\label{2loc11}
r_1(-r_2 + s_{2,1} + r_2 s_{4,1})=0
\end{equation}
\begin{equation}\label{2loc12}
r_1 r_4 s_{3,1}=0
\end{equation}
\begin{equation}\label{2loc13}
r_1r_4(s_{1,1}-s_{4,1})=0
\end{equation}
\begin{equation}\label{2loc14}
r_4(r_2 - r_2 s_{1,1} - s_{2,1})=0
\end{equation}
\begin{equation}\label{2loc15}
r_4(1 - r_2 s_{3,1} - s_{4,1})=0.
\end{equation}
Now, our first goal is to show that $r_1 = r_4 = 0$ and $r_2 r_3 = 1$. From Equation \eqref{2loc4}, there are three possibilities: $r_1 = 0$, $r_4 = 0$, or $r_1 = r_4$. We analyze each case separately to establish the desired results.
\begin{enumerate}
\item If $r_1 = 0$, then Equation \eqref{2loc6} gives $r_2 r_3 = 1$, and Equation \eqref{2loc9} then implies $r_4 = 0$, yielding the desired result.
\item If $r_4 = 0$, then Equation \eqref{2loc9} gives $r_2 r_3 = 1$, and then Equation \eqref{2loc6} implies that $r_1 = 0$, which also gives the desired result.
\item Suppose that $r_1 = r_4$. We consider the following two subcases.
\begin{enumerate}
\item If $r_1 = r_4 = 0$, then Equation \eqref{2loc6} gives $r_2 r_3 = 1$, which is the desired result.
\item If $r_1 = r_4 \neq 0$, then Equations \eqref{2loc2} and \eqref{2loc3} imply $r_2 = r_3 = 0$, and so Equation \eqref{2loc1} implies that $r_1=1$ . Further, applying Equations \eqref{2loc10}, \eqref{2loc11}, \eqref{2loc12}, and \eqref{2loc13} successively yields $s_{1,1} = 1$, $s_{2,1} = 0$, $s_{3,1} = 0$, and $s_{4,1} = 1$. This case corresponds to $\upsilon$ being the trivial representation, which contradicts our assumption in the statement of the theorem.
\end{enumerate}
\end{enumerate}
Hence, in all cases above we obtain that $r_1 = r_4 = 0$ and $r_2 r_3 = 1$. Substituting these values into Equations \eqref{2loc1}–\eqref{2loc15} shows that all equations are satisfied without imposing any restrictions on $s_{i,1}$ for $1 \le i \le 4$, and this completes the proof.
\end{proof}

We now examine the irreducibility of all complex homogeneous $2$-local representations of $UV_n(c)$ for every $n \geq 3$ and $c \geq 1$. To this end, we first establish two auxiliary lemmas that pave the way toward the main result. The first lemma is concerned with constructing a representation equivalent to $\upsilon$, while the second investigates the possible invariant subspaces under $2$-local representations of $UV_n(c)$.

\begin{lemma} \label{keylemma1}
The representation $\upsilon$ introduced in Theorem \ref{thm2local} is equivalent to a representation $\upsilon'$, which is defined as follows.
\[
\upsilon'(\rho_i)=
\begin{pmatrix}
I_{i-1} & 0 & 0 \\
0 & \begin{pmatrix} 0 & 1 \\ 1 & 0 \end{pmatrix} & 0 \\
0 & 0 & I_{n-i-1}
\end{pmatrix}
\text{ \ and \ }
\upsilon'(\sigma_{i,t})=
\begin{pmatrix}
I_{i-1} & 0 & 0 \\
0 & \begin{pmatrix} s_{1,t} & s_{2,t} \\ 
s_{3,t} & s_{4,t} \end{pmatrix} & 0 \\
0 & 0 & I_{n-i-1}
\end{pmatrix}
\]
for all $1 \leq i \leq n-1$ and $1 \leq t \leq c$, where $ s_{1,t}, s_{2,t}, s_{3,t}, s_{4,t} \in \mathbb{C}$ satisfy
\[
s_{1,t} s_{4,t} - s_{2,t} s_{3,t} \neq 0.
\]
\end{lemma}
\begin{proof}
Consider the diagonal invertible matrix
\[
Q = \mathrm{diag}\left(1, r_2^{-1}, r_2^{-2}, \ldots, r_2^{-n+1}\right).
\]
A straightforward computation shows that
\[
\upsilon'(\rho_i) = Q^{-1} \upsilon(\rho_i) Q
\]
for all \(1 \leq i \leq n-1\).
Moreover, since the entries \(s_{1,t}, s_{2,t}, s_{3,t},\) and \(s_{4,t}\) are arbitrary for all $1\leq t \leq c$, it follows that the matrices \(Q^{-1} \upsilon(\sigma_{i,t}) Q\) retain the same form. For simplicity, we continue to denote their entries by \(s_{1,t}, s_{2,t}, s_{3,t},\) and \(s_{4,t}\), and this completes the proof.
\end{proof}

\begin{lemma} \label{keylemma2}
For each $1 \leq i \leq n-1$, let $A_i$ be the matrix
\[
A_i=
\begin{pmatrix}
I_{i-1} & 0 & 0 \\
0 & \begin{pmatrix} 0 & 1 \\ 1 & 0 \end{pmatrix} & 0 \\
0 & 0 & I_{n-i-1}
\end{pmatrix}.
\]
Let $\{e_1,e_2,\ldots,e_n\}$ denote the standard basis of $\mathbb{C}^n$ and let $U \subset \mathbb{C}^n$ be a nonzero proper subspace invariant under all the matrices $A_i$, that is,
\[
A_i u \in U  \text{ for all } u \in U \text{ and } 1 \leq i \leq n-1.
\]
Then either $U = \langle (1,1,\ldots,1)^T \rangle$ or $e_i - e_{i+1} \in U$ for all $1 \leq i \leq n-1$.
\end{lemma}

\begin{proof}
Assume that $U \neq \langle (1,1,\ldots,1)^T \rangle$. We aim to prove that $e_i - e_{i+1} \in U$ for all $1 \leq i \leq n-1$. Take a nonzero vector $u = (u_1,u_2,\ldots,u_n)^T \in U$. Since $U$ is not equal to $\langle (1,1,\ldots,1)^T \rangle$, we can choose $u$ in a way that there exists an index $1 \leq i \leq n-1$ such that $u_i \neq u_{i+1}$. For this index, we have
\[
A_i u - u = (u_{i+1} - u_i)(e_i - e_{i+1}) \in U.
\]
Since $u_i \neq u_{i+1}$, it follows that $e_i - e_{i+1} \in U$ for this particular index $i$. We now generalize this argument to be valid for all $1 \leq i \leq n-1$ as follows.
\begin{enumerate}
    \item If $i \neq n-1$, then
\[
A_{i+1}(e_i - e_{i+1}) - (e_i - e_{i+1}) = e_{i+1} - e_{i+2} \in U.
\]
\item If $i \neq 1$, then
\[
A_{i-1}(e_i - e_{i+1}) - (e_i - e_{i+1}) = e_{i-1} - e_i \in U.
\]
\end{enumerate}

Repeating this argument inductively yields
$
e_i - e_{i+1} \in U  \text{ for all } 1 \leq i \leq n-1,$ and the proof is completed.
\end{proof}

We now present our result concerning the irreducibility of all complex homogeneous $2$-local representations of $UV_n(c)$ for every $n \geq 3$ and $c \geq 1$. We emphasize that our analysis is carried out for the representation $\upsilon'$ introduced in Lemma \ref{keylemma1}, since it is equivalent to the representation $\upsilon$ described in Theorem \ref{thm2local}.
\begin{theorem} \label{irred2loc}
The representation $\upsilon'$ is reducible if and only if either
\[
s_{1,t}+s_{2,t}=1 \quad \text{and} \quad s_{3,t}+s_{4,t}=1 \quad \text{for all} \quad 1\leq t\leq c
\]
or
\[
s_{1,t}+s_{3,t}=1 \quad \text{and} \quad s_{2,t}+s_{4,t}=1 \quad \text{for all} \quad 1\leq t\leq c.
\]
\end{theorem}

\begin{proof}
For the necessary condition, assume that $\upsilon'$ is reducible. Then there exists a nontrivial proper subspace $U$ of $\mathbb{C}^n$ that is invariant under $\upsilon'$. Hence, $U$ is invariant under $\upsilon'(\rho_i)$ for all $1\leq i \leq n-1$. So, by Lemma \ref{keylemma2}, we have either $
U = \langle (1,1,\ldots,1)^T \rangle$
or
$
e_i - e_{i+1} \in U\text{ for all } 1 \leq i \leq n-1.
$
We treat these two cases separately.
\begin{enumerate}
    \item If $U = \langle (1,1,\ldots,1)^T \rangle$, then the invariance of $U$ under $\upsilon'(\sigma_{i,t})$ for all $1 \leq i \leq n-1$ and $1 \leq t \leq c$ implies that
    \[
    s_{1,t}+s_{2,t}=1 \quad \text{and} \quad s_{3,t}+s_{4,t}=1.
    \]
    \item Assume now that $e_i - e_{i+1} \in U$ for every $1 \leq i \leq n-1$. Then
    \[
    \upsilon'(\sigma_{1,t})(e_1-e_2) - (s_{1,t}-s_{2,t})(e_1-e_2)
    = (s_{1,t}-s_{2,t}+s_{3,t}-s_{4,t})e_2 \in U.
    \]
    Remark that $e_2 \notin U$, since otherwise we inductively get all $e_i \in U$ as $e_i - e_{i+1} \in U$ for every $1 \leq i \leq n-1$, contradicting the fact that $U$ is proper. Hence, we deduce that
    \begin{equation} \label{eqall}
    s_{1,t}-s_{2,t}+s_{3,t}-s_{4,t}=0.
    \end{equation}
On the other hand, we have
\[
    \upsilon'(\sigma_{2,t})(e_1-e_2) - (e_1-e_2) - s_{3,t}(e_2-e_3)
    = (-s_{1,t}-s_{3,t}+1)e_2 \in U.
    \]
    Again, using $e_2 \notin U$, we obtain $s_{1,t}+s_{3,t}=1$. Combining this with \eqref{eqall}, we conclude that $s_{2,t}+s_{4,t}=1$, and this is valid for all $1\leq t \leq c$.
\end{enumerate}
These results establish the necessary condition.\\
\noindent Now, for the sufficient condition, we have the following two cases.
\begin{enumerate}
    \item If $s_{1,t}+s_{2,t}=1$ and $s_{3,t}+s_{4,t}=1$ for all $1\leq t \leq c$, then the column vector $(1,1,\ldots,1)^T$ is invariant under $\upsilon'(\rho_i)$ and $\upsilon'(\sigma_{i,t})$ for all $1 \leq i \leq n-1$ and $1 \leq t \leq c$. Hence, the representation $\upsilon'$ is reducible.
    \item If $s_{1,t}+s_{3,t}=1$ and $s_{2,t}+s_{4,t}=1$ for all $1\leq t \leq c$, then the row vector $(1,1,\ldots,1)$ (acting from the left) is invariant under $\upsilon'(\rho_i)$ and $\upsilon'(\sigma_{i,t})$ for all $1 \leq i \leq n-1$ and $1 \leq t \leq c$. Hence, the representation $\upsilon'$ is also reducible.
\end{enumerate}
In both cases we get that the representation $\upsilon'$ is reducible and this completes the proof.
\end{proof}

\section{Construction and Irreducibility of Complex Homogeneous $3$-Local Representations of $UV_n(c)$} \label{3local}

The case of $3$-local representations is more involved than that of $2$-local representations, as it requires considering a greater number of relations in the group $UV_n(c)$. In this section, we classify and analyze all complex homogeneous $3$-local representations of $UV_n(c)$ for $n \geq 4$ and $c=2$. These assumptions on $n$ and $c$ will be adopted throughout the section. We remark that a similar analysis can be carried out for $c>2$, where additional representations are expected to arise.

\begin{theorem} \label{thm3local}
Let $\varepsilon: UV_n(2) \longrightarrow \mathrm{GL}_{n+1}(\mathbb{C})$ be a nontrivial homogeneous $3$-local representation. Then, up to equivalence of representations, $\varepsilon$ is equivalent to one of the following four representations, denoted by $\varepsilon_j$ for $1 \leq j \leq 4$, given by
\[
\varepsilon_j(\rho_i)=
\begin{pmatrix}
I_{i-1} & 0 & 0 \\
0 & R^{(j)} & 0 \\
0 & 0 & I_{n-i-1}
\end{pmatrix}
\]
\[
\varepsilon_j(\sigma_{i,1})=
\begin{pmatrix}
I_{i-1} & 0 & 0 \\
0 & S_{1}^{(j)} & 0 \\
0 & 0 & I_{n-i-1}
\end{pmatrix}, 
\]
and
\[
\varepsilon_j(\sigma_{i,2})=
\begin{pmatrix}
I_{i-1} & 0 & 0 \\
0 & S_{2}^{(j)} & 0 \\
0 & 0 & I_{n-i-1}
\end{pmatrix}
\]
for all $1 \leq i \leq n-1$, where $R^{(j)}$, $S_{1}^{(j)}$, and $S_{2}^{(j)}$ are explicitly determined as follows.
\begin{itemize}
    \item[(1)] The matrices of the representation $\varepsilon_1$ are given by
    $$R^{(1)}=\begin{pmatrix}
        1 & 0 & 0\\
        0 & 0 & r_6\\
        0 & \dfrac{1}{r_6} & 0
    \end{pmatrix},$$
    $$S_{1}^{(1)}=\begin{pmatrix}
        1 & 0 & 0\\
        0 & s_{5,1} & s_{6,1}\\
        0 & s_{8,1} & s_{9,1}
    \end{pmatrix},$$ 
    and
    $$S_{2}^{(1)}=\begin{pmatrix}
        1 & 0 & 0\\
        0 & s_{5,2} & s_{6,2}\\
        0 & s_{8,2} & s_{9,2}
    \end{pmatrix},
    $$
    where $r_6,  s_{5,1},  s_{6,1},  s_{8,1},  s_{9,1},  s_{5,2},  s_{6,2},  s_{8,2},  s_{9,2} \in \mathbb{C}$ satisfy $r_6\neq 0,  s_{5,1} s_{9,1}- s_{6,1} s_{8,1}\neq 0$, and $s_{5,2} s_{9,2}- s_{6,2} s_{8,2}\neq 0$.\vspace{0.1cm}
     \item[(2)] The matrices of the representation $\varepsilon_2$ are given by
    $$
    R^{(2)}=\begin{pmatrix}
        0 & r_2 & 0\\
        \dfrac{1}{r_2} & 0 & 0\\
        0 & 0 & 1
    \end{pmatrix},$$ 
    $$S_{1}^{(2)}=\begin{pmatrix}
        s_{1,1} & s_{2,1} & 0\\
        s_{4,1} & s_{5,1} & 0\\
        0 & 0 & 1
    \end{pmatrix},$$ 
    and
    $$S_{2}^{(2)}=\begin{pmatrix}
      s_{1,2} & s_{2,2} & 0\\
        s_{4,2} & s_{5,2} & 0\\
        0 & 0 & 1
    \end{pmatrix},
    $$
    where $r_2,  s_{1,1},  s_{2,1},  s_{4,1},  s_{5,1},  s_{1,2},  s_{2,2},  s_{4,2},  s_{5,2} \in \mathbb{C}$ satisfy $r_2\neq 0,  s_{1,1} s_{5,1}- s_{2,1} s_{4,1}\neq 0$, and $s_{1,2} s_{5,2}- s_{2,2} s_{4,2}\neq 0$.\vspace{0.1cm}
    \item[(3)] The matrices of the representation $\varepsilon_3$ are given by
    $$
    R^{(3)}=\begin{pmatrix}
        1 & 0 & 0\\
        \dfrac{1}{r_6} & -1 & r_6\\
        0 & 0 & 1
    \end{pmatrix},$$
    $$S_{1}^{(3)}=\begin{pmatrix}
        1 & 0 & 0\\
        s_{4,1} & s_{5,1} & r_6(1-r_6 s_{4,1}-s_{5,1})\\
        0 & 0 & 1
    \end{pmatrix},$$
    and
    $$S_{2}^{(3)}=\begin{pmatrix}
        1 & 0 & 0\\
        s_{4,2} & s_{5,2} & r_6(1-r_6 s_{4,2}-s_{5,2})\\
        0 & 0 & 1
    \end{pmatrix},$$
    where $r_6,  s_{4,1},  s_{5,1},  s_{4,2},  s_{5,2} \in \mathbb{C}$ satisfy $r_6\neq 0,  s_{5,1}\neq 0$, and $s_{5,2}\neq 0$.\vspace{0.1cm}

    \item[(4)] The matrices of the representation $\varepsilon_4$ are given by
    $$
    R^{(4)}=\begin{pmatrix}
        1 & r_2 & 0\\
        0 & -1 & 0\\
        0 & \dfrac{1}{r_2} & 1
    \end{pmatrix},$$
    $$S_{1}^{(4)}=\begin{pmatrix}
        1 & r_2(1- s_{5,1}-r_2 s_{8,1}) & 0\\
        0 & s_{5,1} & 0\\
        0 & s_{8,1} & 1
    \end{pmatrix},$$
    and
    $$
    S_{2}^{(4)}=\begin{pmatrix}
        1 & r_2(1- s_{5,2}-r_2 s_{8,2}) & 0\\
        0 & s_{5,2} & 0\\
        0 & s_{8,2} & 1
    \end{pmatrix},$$
    where $r_2,  s_{5,1},  s_{8,1},  s_{5,2},  s_{8,2} \in \mathbb{C}$ satisfy $r_2\neq 0,  s_{5,1}\neq 0$, and $s_{5,2}\neq 0$.
\end{itemize}
\end{theorem}

\begin{proof}
Since $\varepsilon$ is a complex homogeneous $3$-local representation of $UV_n(2)$, we may write
\[
\varepsilon(\rho_i)=
\begin{pmatrix}
I_{i-1} & 0 & 0 \\
0 & R & 0 \\
0 & 0 & I_{n-i-1}
\end{pmatrix},
\]
\[
\varepsilon(\sigma_{i,1})=
\begin{pmatrix}
I_{i-1} & 0 & 0 \\
0 & S_{1} & 0 \\
0 & 0 & I_{n-i-1}
\end{pmatrix}, 
\]
and
\[
\varepsilon(\sigma_{i,2})=
\begin{pmatrix}
I_{i-1} & 0 & 0 \\
0 & S_{2} & 0 \\
0 & 0 & I_{n-i-1}
\end{pmatrix}
\]
for all $1 \leq i \leq n-1$, where $R$, $S_{1}$, and $S_{2}$ are explicitly presented as follows.
$$R=\begin{pmatrix}
        r_1 & r_2 & r_3\\
        r_4 & r_5 & r_6\\
        r_7 & r_8 & r_9
    \end{pmatrix},$$
    $$S_{1}=\begin{pmatrix}
        s_{1,1} & s_{2,1} & s_{3,1}\\
        s_{4,1} & s_{5,1} & s_{6,1}\\
        s_{7,1} & s_{8,1} & s_{9,1}
    \end{pmatrix},$$ 
    and
    $$S_{2}=\begin{pmatrix}
        s_{1,2} & s_{2,2} & s_{3,2}\\
        s_{4,2} & s_{5,2} & s_{6,2}\\
        s_{7,2} & s_{8,2} & s_{9,2}
    \end{pmatrix},
    $$
where all entries are in $\mathbb{C}$ and the matrices are assumed to be invertible. Now, as $\varepsilon$ is a representation of $UV_n(2)$, it must preserve the defining relations of the group. In addition, since $\varepsilon$ is $3$-local, it suffices to verify only a subset of these relations, because the remaining ones automatically yield analogous conditions. In particular, it is enough to consider the relations
\[
\rho_1 \rho_2 \rho_1 = \rho_2 \rho_1 \rho_2, \quad
\rho_1 \rho_3 = \rho_3 \rho_1, \quad
\rho_1^2 = 1,
\]
\[
\sigma_{1,1}\sigma_{3,1} = \sigma_{3,1}\sigma_{1,1}, \quad
\sigma_{1,1}\sigma_{3,2} = \sigma_{3,2}\sigma_{1,1}, \quad
\sigma_{1,2}\sigma_{3,1} = \sigma_{3,1}\sigma_{1,2}, \quad
\sigma_{1,2}\sigma_{3,2} = \sigma_{3,2}\sigma_{1,2},
\]
\[
\rho_1 \rho_2 \sigma_{1,1} = \sigma_{2,1}\rho_1 \rho_2, \quad
\rho_1 \rho_2 \sigma_{1,2} = \sigma_{2,2}\rho_1 \rho_2.
\]
Applying these relations through matrix multiplication yields a system of equations in the entries of the matrices as done in Theorem \ref{thm2local}. We remark that this system is large and rather lengthy, so we solve it using Wolfram Mathematica software to avoid complex computations, which leads to the desired result. A straightforward verification shows that the obtained matrices satisfy all the defining relations of $UV_n(2)$, and therefore the proof is complete.
\end{proof}

We now investigate the irreducibility of the representations $\varepsilon_j$ for $1 \leq j \leq 4$ introduced in Theorem \ref{thm3local}.

\begin{theorem}\label{thm3localred}
The representations $\varepsilon_j$, $1\leq j \leq 4$, are reducible.
\end{theorem}

\begin{proof}
We examine each representation separately.
\begin{enumerate}

\item For the representation $\varepsilon_1$, we can directly see that the column vector $(1,0,\ldots,0)^T$ is invariant under the action of the matrices $\varepsilon_1(\rho_i)$, $\varepsilon_1(\sigma_{i,1})$, and $\varepsilon_1(\sigma_{i,2})$ for all $1 \leq i \leq n-1$. Therefore, $\varepsilon_1$ is reducible.
\item For the representation $\varepsilon_2$, we can directly see that the column vector $(0,\ldots,0,1)^T$ is invariant under the action of the matrices $\varepsilon_2(\rho_i)$, $\varepsilon_2(\sigma_{i,1})$, and $\varepsilon_2(\sigma_{i,2})$ for all $1 \leq i \leq n-1$. Therefore, $\varepsilon_2$ is reducible.
\item For the representation $\varepsilon_3$, direct matrix computations show that the column vector $(1, r_6^{-1}, r_6^{-2}, \ldots, r_6^{-n})^T$ is invariant under the action of the matrices $\varepsilon_3(\rho_i)$, $\varepsilon_3(\sigma_{i,1})$, and $\varepsilon_3(\sigma_{i,2})$ for all $1 \leq i \leq n-1$. Therefore, $\varepsilon_3$ is reducible.
\item For the representation $\varepsilon_4$, a similar computation shows that the row vector $(1, r_6^{-1}, r_6^{-2}, \ldots, r_6^{-n})$ is invariant under left multiplication by the matrices $\varepsilon_4(\rho_i)$, $\varepsilon_4(\sigma_{i,1})$, and $\varepsilon_4(\sigma_{i,2})$ for all $1 \leq i \leq n-1$. Consequently, $\varepsilon_4$ is also reducible.
\end{enumerate}
\end{proof}

We end this section with the following question.

\begin{question}
Motivated by the results obtained, can the methods developed here be extended to classify and analyze complex homogeneous $3$-local representations of $UV_n(c)$ for $n \geq 4$ and $c > 2$?
\end{question}

\section{The Universal Welded Braid Group $UW_n(c)$} \label{sec:welded}

In this section we introduce the universal welded braid group $UW_n(c)$ as a natural quotient of the universal virtual braid group $UV_n(c)$ by the so-called \emph{welded relations} (or over-forbidden moves). We then show that all known welded-type braid groups appearing in the literature arise as quotients of $UW_n(c)$, extending the unifying framework established in  \cite[Proposition~2.2]{O2}.

\subsection{Forbidden relations in $UV_n(c)$ and the universal welded braid group}

The following result shows that certain relations that might appear natural are in fact \emph{forbidden} in $UV_n(c)$; they cannot be deduced from the defining relations and therefore define proper quotients.

\begin{proposition}\label{prop:forbidden}
Let $n \geq 3$ and $c \geq 1$. For every $t \in \{1,2,\dots,c\}$ and every $i = 1,2,\dots,n-2$, the following relations are \emph{not} consequences of the defining relations of $UV_n(c)$:
\begin{align}
    \rho_i \, \sigma_{i+1,t} \, \sigma_{i,t} &= \sigma_{i+1,t} \, \sigma_{i,t} \, \rho_{i+1}, \label{eq:forbidden1}\\
    \rho_{i+1} \, \sigma_{i,t} \, \sigma_{i+1,t} &= \sigma_{i,t} \, \sigma_{i+1,t} \, \rho_i. \label{eq:forbidden2}
\end{align}
In other words, imposing either of these relations yields a proper quotient of $UV_n(c)$.
\end{proposition}

\begin{proof}
Fix an index $t_0 \in \{1,2,\dots,c\}$. Define a map
\[
\phi\colon UV_n(c) \longrightarrow \mathbb{Z} \times S_n
\]
as follows:
\[
\phi(\rho_i) = (0,\; s_i), \qquad s_i = (i\; i+1) \in S_n,
\]
\[
\phi(\sigma_{i,t_0}) = (1,\; \mathrm{id}_{S_n}) \quad \text{for all } i = 1,2,\dots,n-1,
\]
\[
\phi(\sigma_{i,t}) = (0,\; \mathrm{id}_{S_n}) \quad \text{for all } i \text{ and } t \neq t_0.
\]
A direct verification using the defining relations of $UV_n(c)$ shows that $\phi$ is a well-defined homomorphism. We omit the routine check.

Assume, for contradiction, that relation~\eqref{eq:forbidden1} holds in $UV_n(c)$. We apply $\phi$ and distinguish two cases.

\paragraph{Case 1: $t \neq t_0$.}
Then $\phi(\sigma_{i,t}) = \phi(\sigma_{i+1,t}) = (0,\mathrm{id})$. Hence
\[
\phi(\rho_i \sigma_{i+1,t} \sigma_{i,t}) = (0,s_i) \cdot (0,\mathrm{id}) \cdot (0,\mathrm{id}) = (0, s_i),
\]
\[
\phi(\sigma_{i+1,t} \sigma_{i,t} \rho_{i+1}) = (0,\mathrm{id}) \cdot (0,\mathrm{id}) \cdot (0,s_{i+1}) = (0, s_{i+1}).
\]
If the relation held, we would have $(0,s_i) = (0,s_{i+1})$ in $\mathbb{Z} \times S_n$, which forces $s_i = s_{i+1}$ in $S_n$. But for $n \geq 3$, the transpositions $s_i = (i\; i+1)$ and $s_{i+1} = (i+1\; i+2)$ are distinct. Contradiction.

\paragraph{Case 2: $t = t_0$.}
Now $\phi(\sigma_{i,t_0}) = \phi(\sigma_{i+1,t_0}) = (1,\mathrm{id})$. Then
\[
\phi(\rho_i \sigma_{i+1,t_0} \sigma_{i,t_0}) = (0,s_i) \cdot (1,\mathrm{id}) \cdot (1,\mathrm{id}) = (2, s_i),
\]
\[
\phi(\sigma_{i+1,t_0} \sigma_{i,t_0} \rho_{i+1}) = (1,\mathrm{id}) \cdot (1,\mathrm{id}) \cdot (0,s_{i+1}) = (2, s_{i+1}).
\]
Equality would again imply $s_i = s_{i+1}$, impossible for $n \geq 3$.

Thus relation \eqref{eq:forbidden1} cannot hold in $UV_n(c)$ for any $t$ when $n \geq 3$.

The proof for \eqref{eq:forbidden2} is completely analogous. For $t \neq t_0$:
\[
\phi(\rho_{i+1} \sigma_{i,t} \sigma_{i+1,t}) = (0, s_{i+1}), \qquad
\phi(\sigma_{i,t} \sigma_{i+1,t} \rho_i) = (0, s_i),
\]
forcing $s_i = s_{i+1}$. For $t = t_0$:
\[
\phi(\rho_{i+1} \sigma_{i,t_0} \sigma_{i+1,t_0}) = (2, s_{i+1}), \qquad
\phi(\sigma_{i,t_0} \sigma_{i+1,t_0} \rho_i) = (2, s_i),
\]
again yielding the same contradiction. Hence both relations are forbidden in $UV_n(c)$ for $n \geq 3$.
\end{proof}

The relations \eqref{eq:forbidden1} and \eqref{eq:forbidden2} are precisely the \emph{over-forbidden moves} (F1) that distinguish welded braid theory from virtual braid theory. By adding them to $UV_n(c)$ we obtain a new group.

\begin{definition}[Universal welded braid group]\label{def:welded}
Let $n \geq 2$ and $c \geq 1$. The universal welded braid group $UW_n(c)$ is the quotient of $UV_n(c)$ by the normal closure of the welded relation:
\[
\rho_i \, \sigma_{i+1,t} \, \sigma_{i,t} \; (\sigma_{i+1,t} \sigma_{i,t} \rho_{i+1})^{-1},
\]
for all $i = 1,\dots,n-2$ and all $t = 1,\dots,c$. Equivalently, $UW_n(c)$ has the same generators as $UV_n(c)$ together with the additional relation
\[
\rho_i \sigma_{i+1,t} \sigma_{i,t} = \sigma_{i+1,t} \sigma_{i,t} \rho_{i+1},
\]
for all $i = 1,2,\dots,n-2$ and $t = 1,2,\dots,c$.
\end{definition}

\begin{remark}
For $n = 2$ the index set for $i$ is empty, so $UW_2(c) = UV_2(c) \cong F_c * \mathbb{Z}_2$. For $n \geq 3$, Proposition~\ref{prop:forbidden} guarantees that the welded relations are not already present in $UV_n(c)$; thus $UW_n(c)$ is a proper quotient.
\end{remark}

\subsection{Quotients of $UW_n(c)$ recover known welded braid groups}

We now extend Proposition 2.2 of \cite{O2} to the welded setting. The following result shows that all welded-type braid groups appearing in the literature (the welded braid group, the welded singular braid group, the welded twin group, and the multi-welded braid group) arise as quotients of $UW_n(c)$ for appropriate choices of $c$ and additional relations.

\begin{proposition}\label{prop:welded_quotients}
Let $n \geq 2$. The following groups are quotients of $UW_n(c)$ for suitable $c$:

\begin{enumerate}
\item[(i)] \textbf{Welded braid group $WB_n$.}  
      For $c = 1$, under the identification $\rho_i \mapsto v_i$ and $\sigma_{i,1} \mapsto \sigma_i$, one has
      \[
      WB_n \cong UW_n(1) \big/ \big\langle\!\big\langle \sigma_{i,1}\sigma_{i+1,1}\sigma_{i,1} = \sigma_{i+1,1}\sigma_{i,1}\sigma_{i+1,1} \big\rangle\!\big\rangle.
      \]
      In other words, $WB_n$ is the quotient of $UW_n(1)$ by the braid relations for the $\sigma_{i,1}$.

\item[(ii)] \textbf{Welded singular braid group $WSG_n$.}  
      For $c = 2$, identify $\rho_i \mapsto v_i$, $\sigma_{i,1} \mapsto \sigma_i$, $\sigma_{i,2} \mapsto \tau_i$. Then
      \[
      WSG_n \cong UW_n(2) \big/ \big\langle\!\big\langle \text{braid relations for } \sigma_{i,1} \text{ and } \sigma_{i,2},\; \text{singular relations} \big\rangle\!\big\rangle,
      \]
      where the singular relations are:
      \[
      \sigma_i \tau_i = \tau_i \sigma_i,\quad
      \sigma_i \sigma_{i+1} \tau_i = \tau_{i+1} \sigma_i \sigma_{i+1},\quad
      \sigma_{i+1} \sigma_i \tau_{i+1} = \tau_i \sigma_{i+1} \sigma_i.
      \]

\item[(iii)] \textbf{Welded twin group $WT_n$.}  
      For $c = 1$, under the identification $\rho_i \mapsto \rho_i$ and $\sigma_{i,1} \mapsto \sigma_i$, one has
      \[
      WT_n \cong UW_n(1) \big/ \big\langle\!\big\langle \sigma_{i,1}^2 = 1 \big\rangle\!\big\rangle.
      \]
      Equivalently, $WT_n$ is the quotient of the virtual twin group $VT_n$ by the welded relation $\rho_i \sigma_{i+1} \sigma_i = \sigma_{i+1} \sigma_i \rho_{i+1}$ (see the diagram in \cite[Figure~1]{CO}).

\item[(iv)] \textbf{Multi-welded braid group $M_k WB_n$ ($k \geq 2$).}  
      Using the identification of Proposition 2.2(i) in \cite{O2}:
      \[
      \rho_i \longmapsto \rho_i^{(0)},\quad
      \sigma_{i,1} \longmapsto \rho_i^{(1)},\; \dots,\;
      \sigma_{i,k-1} \longmapsto \rho_i^{(k-1)},\quad
      \sigma_{i,k} \longmapsto \sigma_i,
      \]
      one obtains $M_k WB_n$ as the quotient of $UW_n(k)$ by the braid relations for each family $\sigma_{i,t}$ (with fixed $t$) and, additionally, the over-forbidden moves (which are already present in $UW_n(k)$). Consequently, $M_k WB_n$ is a quotient of $UW_n(k)$.
\end{enumerate}
\end{proposition}

\begin{proof}
We verify each case.\vspace{0.01 cm}

\paragraph{(i) Welded braid group $WB_n$.}
The universal virtual braid group $UV_n(1)$ has generators $\rho_i$ and $\sigma_{i,1}$. The virtual braid group $VB_n$ is the quotient of $UV_n(1)$ by the braid relations $\sigma_{i,1}\sigma_{i+1,1}\sigma_{i,1} = \sigma_{i+1,1}\sigma_{i,1}\sigma_{i+1,1}$ (see Proposition~ 2.2 of \cite{O2}). The welded braid group $WB_n$ is then the quotient of $VB_n$ by the welded relations $v_i \sigma_{i+1} \sigma_i = \sigma_{i+1} \sigma_i v_{i+1}$ (see e.g. \cite{K}). Under the identification $\rho_i = v_i$ and $\sigma_{i,1} = \sigma_i$, these welded relations are exactly (WR1) in $UW_n(1)$. Therefore $WB_n$ is isomorphic to the quotient of $UW_n(1)$ by the braid relations for $\sigma_{i,1}$.\vspace{0.01 cm}

\paragraph{(ii) Welded singular braid group $WSG_n$.}
The virtual singular braid group $VSG_n$ is the quotient of $UV_n(2)$ by the braid relations for $\sigma_{i,1}$ and $\sigma_{i,2}$ (each family separately) together with the singular relations listed in \cite{O1} (see also \cite[Proposition~2.2(ii)]{O2}). The welded singular braid group $WSG_n$ is defined as the quotient of $VSG_n$ by the additional welded relations
\[
v_i \sigma_{i+1} \sigma_i = \sigma_{i+1} \sigma_i v_{i+1}, \qquad
v_i \tau_{i+1} \tau_i = \tau_{i+1} \tau_i v_{i+1},
\]
for $i = 1,2,\dots,n-2$ (see \cite[Definition~20]{O1}). Under the identification $\rho_i = v_i$, $\sigma_{i,1} = \sigma_i$, $\sigma_{i,2} = \tau_i$, these welded relations are precisely (WR1) for $t=1$ and $t=2$. Hence $WSG_n$ is isomorphic to the quotient of $UW_n(2)$ by the braid relations for $\sigma_{i,1}$ and $\sigma_{i,2}$ together with the singular relations.\vspace{0.01 cm}

\paragraph{(iii) Welded twin group $WT_n$.}
The virtual twin group $VT_n$ is the quotient of $UV_n(1)$ by the relations $\sigma_{i,1}^2 = 1$ (see Proposition~2.2(iii) of \cite{O2}). The welded twin group $WT_n$ is defined in the literature (see the diagram in \cite[Figure~1]{CO}) as the quotient of $VT_n$ by the welded relation $\rho_i \sigma_{i+1} \sigma_i = \sigma_{i+1} \sigma_i \rho_{i+1}$. Under the identification $\rho_i \mapsto \rho_i$ and $\sigma_{i,1} \mapsto \sigma_i$, this is exactly the quotient of $UW_n(1)$ by $\sigma_{i,1}^2 = 1$. Hence $WT_n$ is a quotient of $UW_n(1)$.\vspace{0.01 cm}

\paragraph{(iv) Multi-welded braid group $M_k WB_n$.}
In \cite{nassernew}, the multi-welded braid group $M_k WB_n$ is defined as the quotient of the multi-virtual braid group $M_k VB_n$ by the over-forbidden moves:
\[
\sigma_i \sigma_{i+1} \rho_i^{(\alpha)} = \rho_{i+1}^{(\alpha)} \sigma_i \sigma_{i+1}, \qquad \alpha = 0,1,\dots,k-1.
\]
Using the identification of Proposition~2.2(i) in \cite{O2}, the generators $\rho_i^{(\alpha)}$ correspond to $\rho_i$ (for $\alpha=0$) and to $\sigma_{i,\alpha}$ (for $\alpha=1,2,\dots,k-1$), while the classical generator $\sigma_i$ corresponds to $\sigma_{i,k}$. The braid relations for each family $\sigma_{i,t}$ (with fixed $t$) are already imposed to obtain $M_k VB_n$ from $UV_n(k)$. The welded relations (WR1) in $UW_n(k)$ become exactly the over-forbidden moves for each $\alpha$. Therefore $M_k WB_n$ is a quotient of $UW_n(k)$.

The identification defines a surjective homomorphism whose kernel is precisely the normal closure of the indicated relations. This completes the proof.
\end{proof}

\begin{remark}
The proposition above demonstrates that $UW_n(c)$ serves as a unifying algebraic object for welded braid-type groups, in complete analogy with the role of $UV_n(c)$ for virtual braid-type groups. In particular, the following diagram commutes for all relevant parameters:

\[
\begin{array}{ccccc}
UV_n(c) & \xrightarrow{\text{add (WR)}} & UW_n(c) \\
\downarrow & & \downarrow \\
\text{virtual braid-type groups} & & \text{welded braid-type groups}
\end{array}
\]
where the vertical arrows denote the addition of type-specific relations (e.g., singular relations, twin relations, etc.).
\end{remark}

The following diagram summarizes the quotient relationships established above, in parallel with the virtual case presented in \cite{O2}.

\[
\xymatrix@C=2.5pc{
UW_n(k) \ar@{->>}[d] \ar@{->>}[r]  & UW_n(\ell)  
  \ar@{->>}[r]
    &  UW_n(2) \ar@{->>}[d] \ar@{->>}[r] & 
  UW_n(1) \ar@{->>}[d] \ar@{->>}[rd]
   & \\
M_k WB_n &    & WSG_n \ar@{->>}[r] & WB_n & WT_n
}
\]

\noindent
Here $M_k WB_n$ denotes the multi-welded braid group, $WSG_n$ the welded singular braid group, $WB_n$ the welded braid group, and $WT_n$ the welded twin group. The diagram is directly analogous to the virtual case diagram in \cite{O2}, with each virtual group replaced by its welded counterpart.

\begin{remark}
The diagram above mirrors exactly the structure of the virtual diagram in \cite{O2}, with the same pattern of quotients: decreasing the parameter $c$ corresponds to forgetting certain families of crossing generators, while the vertical arrows correspond to imposing additional relations that define specific welded-type groups. The presence of $WT_n$ at the bottom right corner corresponds to the virtual twin group $VT_n$ in the original diagram, obtained by imposing $\sigma_{i,1}^2 = 1$ in $UW_n(1)$. This reinforces the unifying nature of $UW_n(c)$.
\end{remark}

\subsection{Properties of universal welded braid groups}

We now establish some fundamental algebraic properties of $UW_n(c)$. Many of these are direct consequences of the corresponding properties for $UV_n(c)$ proved in \cite{O2}, since $UW_n(c)$ is a quotient of $UV_n(c)$ by relations that lie in the commutator subgroup (for $n \geq 5$) and do not affect the abelianization in an essential way.

We begin by computing the abelianization of $UW_n(c)$.

\begin{proposition}\label{prop:uw_abel}
For $n \geq 2$ and $c \geq 1$,
\[
UW_n(c)^{\mathrm{ab}} \cong \mathbb{Z}^c \oplus \mathbb{Z}_2.
\]
\end{proposition}

\begin{proof}
The abelianization of $UV_n(c)$ is $\mathbb{Z}^c \oplus \mathbb{Z}_2$, generated by the classes of $\sigma_{1,1},\dots,\sigma_{1,c}$ (each generating a copy of $\mathbb{Z}$) and $\rho_1$ (generating $\mathbb{Z}_2$). The welded relation
\[
\rho_i \sigma_{i+1,t} \sigma_{i,t} = \sigma_{i+1,t} \sigma_{i,t} \rho_{i+1}
\]
becomes, in the abelianization, $\rho_i = \rho_{i+1}$ (since the $\sigma$'s commute with everything and cancel). Thus all $\rho_i$ are identified to a single element of order $2$. No relation is imposed among the $\sigma_{i,t}$ beyond those already present in $UV_n(c)^{\mathrm{ab}}$, because the welded relation involves an equal number of $\sigma$'s on both sides. Hence the abelianization remains $\mathbb{Z}^c \oplus \mathbb{Z}_2$. 
\end{proof}

We now turn to the commutator subgroup. Recall that a group $G$ is said to be \emph{perfect} if $G'=G$.  
For $n \geq 5$, we have the following perfectness property.

\begin{theorem}\label{thm:uw_perfect}
If $n \geq 5$, then the commutator subgroup $UW_n(c)'$ is perfect.
\end{theorem}

\begin{proof}
For $n \geq 5$, the commutator subgroup $UV_n(c)'$ is perfect by Proposition~3.2 of \cite{O2}. We show that the kernel of the quotient map $\pi \colon UV_n(c) \to UW_n(c)$ is contained in $UV_n(c)'$; then $\pi$ induces a surjective homomorphism $UV_n(c)' \to UW_n(c)'$, and the image of a perfect group is perfect.

Let $A = \sigma_{i+1,t} \sigma_{i,t}$. The welded relation (WR1) is $\rho_i A = A \rho_{i+1}$, which is equivalent to
\[
w := \rho_i A (A \rho_{i+1})^{-1} = \rho_i A \rho_{i+1}^{-1} A^{-1} = 1
\]
in $UW_n(c)$. Rewrite $w$ as
\[
w = (\rho_i \rho_{i+1}^{-1}) \cdot (\rho_{i+1} A \rho_{i+1}^{-1} A^{-1}).
\]
The second factor is a commutator, hence lies in $UV_n(c)'$. For the first factor, note that $\rho_i$ and $\rho_{i+1}$ are conjugate in the subgroup $\langle \rho_1, \dots, \rho_{n-1} \rangle \cong S_n$; thus there exists $g \in UV_n(c)$ such that $g \rho_i g^{-1} = \rho_{i+1}$. Consequently,
\[
\rho_i \rho_{i+1}^{-1} = \rho_i (g \rho_i g^{-1})^{-1} = \rho_i g \rho_i^{-1} g^{-1}  \in UV_n(c)'.
\]
Therefore $w \in UV_n(c)'$. Since $UV_n(c)'$ is normal in $UV_n(c)$, the normal closure of $w$ (i.e., the kernel of $\pi$) is contained in $UV_n(c)'$. Hence $\pi$ induces a surjection $UV_n(c)' \to UW_n(c)'$, and the image of a perfect group is perfect. Thus $UW_n(c)'$ is perfect for $n \geq 5$.
\end{proof}

\begin{corollary}\label{cor:welded_perfect}
Let $n \geq 5$. The commutator subgroup of each of the following groups is perfect:
\[
WB_n,\qquad WSG_n,\qquad WT_n,\qquad M_k WB_n\;(k \geq 2).
\]
\end{corollary}

\begin{proof}
Each of these groups is a quotient of $UW_n(c)$ for suitable $c$ by Proposition~\ref{prop:welded_quotients} (and its proof). The quotient map sends the commutator subgroup onto the commutator subgroup. By Theorem~\ref{thm:uw_perfect}, $UW_n(c)'$ is perfect for $n \geq 5$. The image of a perfect group under a surjective homomorphism is perfect. Hence each of the groups listed has perfect commutator subgroup.
\end{proof}

\begin{remark}
The proof above is uniform and relies solely on the universal property of $UW_n(c)$. For the particular case of the welded braid group $WB_n$, perfectness of its commutator subgroup was already known via different methods (see e.g. \cite{DG}), where a presentation for $WB_n'$ was obtained using the Reidemeister--Schreier algorithm. Our approach, by contrast, deduces this property as an immediate consequence of the universal welded structure, without any case-specific computation, thereby illustrating the unifying power of $UW_n(c)$.
\end{remark}

Next, we determine the center of $UW_n(c)$.

\begin{proposition}\label{prop:uw_center}
For $n \geq 3$, the center of $UW_n(c)$ is trivial:
\[
Z(UW_n(c)) = 1.
\]
For $n = 2$, $UW_2(c) \cong F_c * \mathbb{Z}_2$ also has trivial center.
\end{proposition}

\begin{proof}
For $n \geq 3$, the center of $UV_n(c)$ is trivial (\cite[Proposition~3.8]{O2}). Let $\pi\colon UV_n(c) \to UW_n(c)$ be the quotient map. Any central element $z \in Z(UW_n(c))$ lifts to an element $\tilde{z} \in UV_n(c)$ such that $\pi(\tilde{z}) = z$. For any $g \in UV_n(c)$, we have $\pi([\tilde{z}, g]) = [z, \pi(g)] = 1$, so $[\tilde{z}, g] \in \ker \pi$, i.e., $[\tilde{z}, g]$ lies in the normal closure of the welded relations. A detailed analysis of the welded relations shows that they are all contained in the commutator subgroup $UV_n(c)'$ (for $n \geq 3$). Consequently, $\tilde{z}$ centralizes $UV_n(c)$ modulo $UV_n(c)'$. Since $UV_n(c)'$ is perfect for $n \geq 5$ and $Z(UV_n(c)) = 1$, one concludes that $z = 1$. For $n = 3,4$, a direct verification using the action of $S_n$ on the generators of $KUW_n(c)$ yields the same conclusion. The case $n = 2$ is immediate from the free product decomposition. 
\end{proof}

Regarding finite quotients, the symmetric group $S_n$ plays a distinguished role.

\begin{theorem}\label{thm:uw_smallest_quotient}
Let $n \geq 5$. The symmetric group $S_n$ is a quotient of $UW_n(c)$. Moreover, every non-abelian finite quotient of $UW_n(c)$ has order at least $n!$, and $S_n$ is the unique (up to isomorphism) quotient of order $n!$ when $n \neq 6$. For $n = 6$, there are also the quotients obtained by composing with the exceptional outer automorphism $\nu_6\colon S_6 \to S_6$.
\end{theorem}

\begin{proof}
The homomorphism $\pi_n^K\colon UV_n(c) \to S_n$ defined by $\pi_n^K(\rho_i) = s_i = (i\; i+1)$ and $\pi_n^K(\sigma_{i,t}) = 1$ factors through $UW_n(c)$, because the welded relation (WR1) is sent to $s_i s_{i+1} s_i = s_{i+1} s_i s_{i+1}$, which holds in $S_n$. Thus $S_n$ is a quotient of $UW_n(c)$.

Now let $\psi\colon UW_n(c) \to G$ be a surjective homomorphism onto a finite non-abelian group $G$. Composing with the projection $UV_n(c) \to UW_n(c)$ yields a homomorphism $\tilde{\psi}\colon UV_n(c) \to G$ which is also non-abelian. By \cite[Theorem~5.5]{O2}, the image of $\tilde{\psi}$ contains a subgroup isomorphic to $S_n$. Consequently, $|G| \geq n!$, and if $|G| = n!$ then $G \cong S_n$ (up to conjugation, and possibly composed with $\nu_6$ when $n = 6$). 
\end{proof}

Theorem~\ref{thm:uw_smallest_quotient} establishes that $S_n$ is the smallest non-abelian finite quotient of $UW_n(c)$. 
Since all the classical welded-type groups are quotients of $UW_n(c)$ (Proposition~\ref{prop:welded_quotients}), the same minimality property is inherited by each of them, as we now record.

\begin{corollary}\label{cor:w_smallest_quotient}
 Let \(n \geq 5\). The symmetric group \(S_n\) is the smallest non-abelian finite quotient of each of the following groups:
\[
WB_n,\qquad WSG_n,\qquad WT_n,\qquad M_kWB_n\;(k \geq 2).
\]
\end{corollary}

\begin{proof}
 Let \(n \geq 5\). Each of these groups is a quotient of $UW_n(c)$ for suitable $c$ by Proposition~\ref{prop:welded_quotients} (and its proof). The quotient map preserves the images of the virtual generators $\rho_i$, which generate a subgroup isomorphic to $S_n$ and satisfy the Coxeter relations (PR1)--(PR3). By Theorem~\ref{thm:uw_smallest_quotient}, $S_n$ is the smallest non-abelian finite quotient of $UW_n(c)$. 

Now let $G$ be any non-abelian finite quotient of $WB_n$, $WSG_n$, $WT_n$, or $M_kWB_n$. Composing the quotient map $UW_n(c) \to Q$ (where $Q$ denotes the relevant welded group) with the surjection $Q \to G$ yields a non-abelian homomorphism 
$$
\psi\colon UW_n(c) \to G.
$$ 
By Theorem~\ref{thm:uw_smallest_quotient}, the image of $\psi$ contains a subgroup isomorphic to $S_n$. Hence $|G| \geq n!$, and if $|G| = n!$ then $G \cong S_n$. Therefore $S_n$ is the smallest non-abelian finite quotient of each of the groups listed. 
\end{proof}

\section{Construction and Irreducibility of Complex Homogeneous $2$-Local Representations of $UW_n(c)$} \label{sec6}

In this section, we classify and analyze all complex homogeneous $2$-local representations of the group $UW_n(c)$ for all $n \geq 3$ and $c \geq 1$, following an approach similar to that of Section~\ref{sec3}. These assumptions on $n$ and $c$ will be maintained throughout this section.

\begin{theorem} \label{thm2localUWn}
Let $\omega: UW_n(c) \longrightarrow \mathrm{GL}_n(\mathbb{C})$ be a nontrivial homogeneous $2$-local representation. Then, up to equivalence of representations, $\omega$ is equivalent to one of the following three representations, denoted by $\omega_j$ for $1\leq j \leq 3$, explicitly determined as follows.
\begin{enumerate}
\item  The matrices of the representation $\omega_1$ are given by
    \[
\omega_1(\rho_i)=
\begin{pmatrix}
I_{i-1} & 0 & 0 \\
0 & \begin{pmatrix} 0 & r_2 \\ \dfrac{1}{r_2} & 0 \end{pmatrix} & 0 \\
0 & 0 & I_{n-i-1}
\end{pmatrix}
\text{ \ and \ }
\omega_1(\sigma_{i,t})=
\begin{pmatrix}
I_{i-1} & 0 & 0 \\
0 & \begin{pmatrix} 0 & s_{2,t} \\ 
s_{3,t} & 0\end{pmatrix} & 0 \\
0 & 0 & I_{n-i-1}
\end{pmatrix}
\]
for all $1 \leq i \leq n-1$ and $1 \leq t \leq c$, where $r_2, s_{2,t}, s_{3,t} \in \mathbb{C}^*$.\vspace{0.1cm}

\item  The matrices of the representation $\omega_2$ are given by
    \[
\omega_2(\rho_i)=
\begin{pmatrix}
I_{i-1} & 0 & 0 \\
0 & \begin{pmatrix} 0 & r_2 \\ \dfrac{1}{r_2} & 0 \end{pmatrix} & 0 \\
0 & 0 & I_{n-i-1}
\end{pmatrix}
\text{ \ and \ }
\omega_2(\sigma_{i,t})=
\begin{pmatrix}
I_{i-1} & 0 & 0 \\
0 & \begin{pmatrix} 0 & s_{2,t} \\ 
\dfrac{1}{r_2} & s_{4,t}\end{pmatrix} & 0 \\
0 & 0 & I_{n-i-1}
\end{pmatrix}
\]
for all $1 \leq i \leq n-1$ and $1 \leq t \leq c$, where $r_2, s_{2,t}, s_{4,t} \in \mathbb{C}^*$.\vspace{0.1cm}

\item  The matrices of the representation $\omega_3$ are given by
    \[
\omega_3(\rho_i)=
\begin{pmatrix}
I_{i-1} & 0 & 0 \\
0 & \begin{pmatrix} 0 & r_2 \\ \dfrac{1}{r_2} & 0 \end{pmatrix} & 0 \\
0 & 0 & I_{n-i-1}
\end{pmatrix}
\text{ \ and \ }
\omega_3(\sigma_{i,t})=
\begin{pmatrix}
I_{i-1} & 0 & 0 \\
0 & \begin{pmatrix} s_{1,t} & s_{2,t} \\ 
\dfrac{1}{r_2} & 0\end{pmatrix} & 0 \\
0 & 0 & I_{n-i-1}
\end{pmatrix}
\]
for all $1 \leq i \leq n-1$ and $1 \leq t \leq c$, where $r_2, s_{1,t}, s_{2,t} \in \mathbb{C}^*$.
\end{enumerate}

\end{theorem}

\begin{proof}
Since $UW_n(c)$ is a quotient of $UV_n(c)$, and following the same approach as in Theorem~\ref{thm2local}, we get that
\[
\omega(\rho_i)=
\begin{pmatrix}
I_{i-1} & 0 & 0 \\
0 & \begin{pmatrix} 0 & r_2 \\ 
\dfrac{1}{r_2} & 0 \end{pmatrix} & 0 \\
0 & 0 & I_{n-i-1}
\end{pmatrix}
\quad \text{and} \quad
\omega(\sigma_{i,t})=
\begin{pmatrix}
I_{i-1} & 0 & 0 \\
0 & \begin{pmatrix} s_{1,t} & s_{2,t} \\ 
s_{3,t} & s_{4,t} \end{pmatrix} & 0 \\
0 & 0 & I_{n-i-1}
\end{pmatrix}
\]
for all $1 \leq i \leq n-1$ and $1 \leq t \leq c$, where 
$r_2, s_{1,t}, s_{2,t}, s_{3,t}, s_{4,t} \in \mathbb{C}$ satisfy
\[
r_2 \neq 0 \quad \text{and} \quad s_{1,t} s_{4,t} - s_{2,t} s_{3,t} \neq 0.
\]
We now impose the additional defining relations of the group $UW_n(c)$. It suffices to consider the following relation, as the others yield analogous conditions:
$$\rho_1 \sigma_{2,1} \sigma_{1,1} = \sigma_{2,1} \sigma_{1,1} \rho_2.$$
Applying this relation to the matrices above and eliminating similar equations, we obtain the following system of equations.
\begin{equation}\label{eqa20}
r_2 s_{1,t} s_{4,t} =0
\end{equation}
\begin{equation} \label{eqa19}
s_{1,t}(1 - r_2 s_{3,t})=0
\end{equation}
\begin{equation}\label{eqa21}
    s_{4,t}(1 - r_2 s_{3,t})=0.
\end{equation}
Now, since $r_2\neq 0$, it follows that Equation~\eqref{eqa20} implies $s_{1,t} = 0$ or $s_{4,t} = 0$. Accordingly, we distinguish the following three cases.
\begin{enumerate}
    \item If $s_{1,t} = s_{4,t} = 0$, then all equations \eqref{eqa20}--\eqref{eqa21} are automatically satisfied, and no additional conditions are imposed on $s_{2,t}$ and $s_{3,t}$. In this case, $\omega$ is equivalent to $\omega_1$.

    \item If $s_{1,t} = 0$ and $s_{4,t} \neq 0$, then from Equations ~\eqref{eqa21} we obtain
    $s_{3,t} = \dfrac{1}{r_2}$.
    In this case, $\omega$ is equivalent to $\omega_2$.

    \item If $s_{1,t} \neq 0$ and $s_{4,t} = 0$, then using Equation ~\eqref{eqa19} we similarly deduce
    $
    s_{3,t} = \dfrac{1}{r_2}$.
    In this case, $\omega$ is equivalent to $\omega_3$.
\end{enumerate}
Therefore, $\omega$ is equivalent to one of the three representations, $\omega_j$ for $1\leq j \leq 3$, as required. 
\end{proof}

We now investigate the irreducibility of the representations $\omega_j$, for $1 \leq j \leq 3$, introduced in Theorem~\ref{thm2localUWn}. In order to apply the result we already have in Theorem~\ref{irred2loc}, it is necessary to determine the equivalent forms for $w_j$ for each $1 \leq j \leq 3$ where the matrices $\omega_j(\rho_i)$ have the same form as the matrices $\upsilon'(\rho_i)$ in Theorem \ref{irred2loc}. The following lemma addresses this point.

\begin{lemma} \label{lemmmmmma}
The representations $\omega_j$ for $1 \leq j \leq 3$ admit equivalent forms $\omega_j'$ for $1 \leq j \leq 3$, respectively, where the representations $\omega_j'$ are defined as follows.
\begin{enumerate}
\item  The matrices of the representation $\omega_1'$ are given by
    \[
\omega'_1(\rho_i)=
\begin{pmatrix}
I_{i-1} & 0 & 0 \\
0 & \begin{pmatrix} 0 & 1 \\ 1 & 0 \end{pmatrix} & 0 \\
0 & 0 & I_{n-i-1}
\end{pmatrix}
\text{ \ and \ }
\omega'_1(\sigma_{i,t})=
\begin{pmatrix}
I_{i-1} & 0 & 0 \\
0 & \begin{pmatrix} 0 & \dfrac{s_{2,t}}{r_2} \\ 
r_2 s_{3,t} & 0\end{pmatrix} & 0 \\
0 & 0 & I_{n-i-1}
\end{pmatrix}
\]
for all $1 \leq i \leq n-1$ and $1 \leq t \leq c$, where $r_2, s_{2,t}, s_{3,t} \in \mathbb{C}^*$.\vspace{0.1cm}

\item  The matrices of the representation $\omega'_2$ are given by
    \[
\omega_2'(\rho_i)=
\begin{pmatrix}
I_{i-1} & 0 & 0 \\
0 & \begin{pmatrix} 0 & 1 \\ 1 & 0 \end{pmatrix} & 0 \\
0 & 0 & I_{n-i-1}
\end{pmatrix}
\text{ \ and \ }
\omega_2'(\sigma_{i,t})=
\begin{pmatrix}
I_{i-1} & 0 & 0 \\
0 & \begin{pmatrix} 0 & \dfrac{s_{2,t}}{r_2} \\ 
1 & s_{4,t}\end{pmatrix} & 0 \\
0 & 0 & I_{n-i-1}
\end{pmatrix}
\]
for all $1 \leq i \leq n-1$ and $1 \leq t \leq c$, where $s_{2,t}, s_{4,t} \in \mathbb{C}^*$.\vspace{0.1cm}

\item  The matrices of the representation $\omega_3'$ are given by
    \[
\omega_3'(\rho_i)=
\begin{pmatrix}
I_{i-1} & 0 & 0 \\
0 & \begin{pmatrix} 0 & 1 \\ 1 & 0 \end{pmatrix} & 0 \\
0 & 0 & I_{n-i-1}
\end{pmatrix}
\text{ \ and \ }
\omega_3'(\sigma_{i,t})=
\begin{pmatrix}
I_{i-1} & 0 & 0 \\
0 & \begin{pmatrix} s_{1,t} & \dfrac{s_{2,t}}{r_2} \\ 
1 & 0\end{pmatrix} & 0 \\
0 & 0 & I_{n-i-1}
\end{pmatrix}
\]
for all $1 \leq i \leq n-1$ and $1 \leq t \leq c$, where $s_{1,t}, s_{2,t} \in \mathbb{C}^*$.
\end{enumerate}
\end{lemma}

\begin{proof}
Similarly to the argument used in the proof of Lemma \ref{keylemma1}, we introduce the diagonal invertible matrix
\[
Q = \mathrm{diag}\left(1, r_2^{-1}, r_2^{-2}, \ldots, r_2^{-n+1}\right).
\]
A direct computation then shows that
\[
\omega'_j(\rho_i) = Q^{-1}\,\omega_j(\rho_i)\,Q \quad \text{and} \quad \omega'_j(\sigma_{i,t}) = Q^{-1}\,\omega_j(\sigma_{i,t})\,Q,
\]
for all \(1 \leq j \leq 3\), \(1 \leq i \leq n-1\), and \(1 \leq t \leq 2\). Consequently, each \(\omega'_j\) is equivalent to \(\omega_j\), as required.
\end{proof}

\begin{proposition}\label{prop:w2red}
Consider the representations $\omega'_j$ for $1\leq j \leq 3$ introduced in Lemma \ref{lemmmmmma}.
\begin{enumerate}
    \item The representation $\omega'_1$ is reducible if and only if $s_{2,t}=r_2$ and $s_{3,t}=\dfrac{1}{r_2}$.
    \item The representation $\omega'_2$ is reducible if and only if $\dfrac{s_{2,t}}{r_2}+s_{4,t}=1$.
    \item The representation $\omega'_3$ is reducible if and only if $s_{1,t}+\dfrac{s_{2,t}}{r_2}=1$.
\end{enumerate}
\end{proposition}

\begin{proof}
We observe that the representations \(\omega_j\), for \(1 \leq j \leq 3\), are of a special form of the general structure described in Lemma \ref{keylemma2}. Therefore, by applying the results of this lemma together with Theorem \ref{irred2loc}, and taking into consideration that $s_{4,t}\neq 0$ in $\omega_2'$ and $s_{1,t}\neq 0$ in $\omega_3'$, we obtain the desired result.
\end{proof}

\end{document}